\documentclass{scrartcl}
\usepackage[margin=3cm]{geometry}
\usepackage[utf8]{inputenc}
\usepackage{amsmath, amsthm, amssymb, amsfonts}
\usepackage[T1]{fontenc}
\usepackage{tikz-cd} 
\usepackage[english]{babel} 
\usepackage{url}
\usepackage{amsmath}
\usepackage{tcolorbox}
\usepackage{xcolor}
\usepackage{hyperref}
\theoremstyle{plain}
\newtheorem{thm}{Theorem}

\newtheorem{lemma}[thm]{Lemma}

\theoremstyle{definition}
\newtheorem{rmk}{Remark}

\newtheorem{question}{Question}

\usepackage{upgreek}
\usepackage{mathtools}
\usepackage{filecontents}
\DeclareMathOperator{\br}{Br}
\DeclareMathOperator{\inv}{inv}

\DeclareMathOperator{\spec}{Spec}
\DeclareMathOperator{\rsw}{rsw}
\DeclareMathOperator{\Pic}{Pic}
\DeclareMathOperator{\Ev}{Ev}
\DeclareMathOperator{\im}{im}
\DeclareMathOperator{\h}{H}

\DeclareMathOperator{\fil}{fil}
\newcommand{\numberset}{\mathbb}
\newcommand{\p}{\numberset{P}}
\newcommand{\F}{\numberset{F}}

\newcommand{\Z}{\numberset{Z}}
\newcommand{\Q}{\numberset{Q}}
\newcommand{\R}{\numberset{R}}

\newcommand{\Os}{\mathcal{O}}
\newcommand{\Ad}{\mathbf{A}}

\newcommand{\A}{\mathcal{A}}
\newcommand{\proj}{\text{Proj}}
\newcommand{\et}{\text{ét}}
\usepackage{dsfont}
\usepackage{commath}

\makeatletter

 \title{\LARGE An example of a Brauer--Manin obstruction to weak approximation at a prime with good reduction}
\date{}
\author{Margherita Pagano}

\begin{document}

\maketitle
\begin{abstract}

    \noindent\textbf{Abstract:} Following Bright and Newton, we construct an explicit K$3$ surface over the rational numbers having good reduction at 2, and for which 2 is the only prime at which weak approximation is obstructed.
\end{abstract}
\section{Introduction}

Let $k$ be a number field and $\Ad_k$ be the ring of adèles of $k$, i.e.\ the restricted product of $k_\nu$ for all places $\nu$ of $k$, taken with respect to the rings of integers $\Os_\nu\subseteq k_\nu$. Let $X$ be a smooth, proper, geometrically irreducible variety over $k$. In order to study the rational points on $X$ it is useful to look at the image of $X(k)$ in the set of the adèlic points $X(\Ad_k)$. More precisely, Manin \cite{Manin} has shown that there exists a pairing 
$$\br(X)\times X(\Ad_k)\rightarrow \Q/\Z$$
such that the rational points of $X$ lie in the image of the right kernel of the pairing, denoted by $X(\Ad_k)^{\br}$. If $X(\Ad_k)^{\br}$ is not equal to the whole $X(\Ad_k)$ we say that there is a \emph{Brauer--Manin obstruction to weak approximation} on $X$. 

In this paper we follow the ideas presented in \cite{https://doi.org/10.48550/arxiv.2009.03282} to construct an example of a K$3$ surface $X$ over the rational numbers with a Brauer--Manin obstruction to weak approximation arising from a prime of good ordinary reduction. More precisely, there exist an element $\A\in \br(X)$ and a prime $p$ of good ordinary reduction such that the evaluation map $|\A|\colon X(\Q_p)\rightarrow\br(\Q_p)$ is non-constant.

\vspace{5mm}

Let ${X}\subseteq \p^3_\Q$ be the projective K$3$ surface defined by the equation 
\begin{equation}\label{eq}
    x^3y+y^3z+z^3w+w^3x+xyzw=0.
\end{equation}
\begin{thm}\label{thm}
The class of the quaternion algebra 
$$\A=\left(\frac{z^3+w^2x+xyz}{x^3},-\frac{z}{x}\right)\in \br k(X)$$
defines an element in $\br(X)$. The evaluation map $|\A|\colon {X}(\Q_2)\rightarrow \br(\Q_2)$ is non-constant, and therefore gives an obstruction to weak approximation on $X$. Finally, $X(\Q)$ is not dense in $X(\Q_2)$, with respect to the $2$-adic topology. 
\end{thm}

This theorem shows, with a concrete example, what was already predicted by Bright and Newton in \cite[Theorem C]{https://doi.org/10.48550/arxiv.2009.03282}. Indeed, they prove that, for a smooth, proper variety $V$ over a number field $L$ such that $\h^0(V,\Omega^2_V)\ne 0$, every prime of good ordinary reduction is involved in a Brauer--Manin obstruction over an extension of the base field. The example provided in this article is optimal, in the following sense: we build an element $\A\in \br(X)[2]$ such that the evaluation map associated to it is non-constant without the need to take an algebraic extension of $\Q$. Moreover, as pointed out in \cite[Remark $11.5$]{https://doi.org/10.48550/arxiv.2009.03282}, when we are dealing with K$3$ surfaces defined over the rational numbers, the only prime with good reduction that can play a role in the obstruction to weak approximation is the prime $2$.  

The following example and, more generally, the result proven by Bright and Newton give a negative answer to the following question, asked by Swinnerton-Dyer \cite[Question 1]{colliotskogoodred}.
\begin{question}\label{questionSD}
Let $k$ be a number field and let $S$ be a finite set of
places of $k$ containing the Archimedean places. Let $\mathcal{V}_S$ be a smooth projective $\Os_S$-scheme with geometrically integral fibres, and let $V/k$ be the generic fibre. Assume that $\Pic(\overline{V})$ is finitely generated and torsion-free. Swinnerton-Dyer asks if there is an open and closed $Z\subseteq \prod_{\nu \in S} V(k_\nu)$ such that 
$$V(\Ad_k)^{\br}=Z\times \prod_{\nu \notin S} V(k_\nu).$$
Roughly speaking, is it true that the Brauer--Manin obstruction involves only the places of bad reduction and the Archimedean places?
\end{question}

Finally, we point out that the element $\A$ defined in Theorem \ref{thm} has to be a transcendental element in $\br(X)$. Indeed, Colliot-Thélène and Skorobogatov proved \cite[Lemma 2.2]{colliotskogoodred} that for every element in the algebraic Brauer group the associated evaluation map at a prime with good reduction has to be constant. 

In general, let $V$ be a variety over a field $k$, $\overline{k}$ an algebraic closure of $k$ and $\overline{V}$ the base change of $V$ to $\overline{k}$, i.e.\ $\overline{V}:=V\times_k \overline{k}$; the algebraic and transcendental Brauer groups of $V$ are defined, respectively, as the kernel and the image of the natural map $\br(V)\rightarrow \br(\overline{V})$. The algebraic Brauer group of $V$ is denoted by $\br_1(V)$.

For curves and surfaces with negative Kodaira dimension we have $\br(V)=\br_1(V)$. Hence, K$3$ surfaces are among the first example of varieties where the transcendental Brauer group is potentially non-trivial. However, this is not always the case: for example in \cite{IerSkoZar} they show that, under certain conditions, the whole Brauer group of a diagonal quartic surface over $\Q$ is algebraic. The first example of a transcendental element in the Brauer group of a K$3$ surface defined over a number field was given by Wittenberg in \cite{Wittenberg}. In particular, Wittenberg constructed a $2$-torsion transcendental element that obstructs weak approximation on the surface. Other examples of $2$-torsion transcendental elements that obstruct weak approximation can be found in \cite{HasVarWAorder2} and \cite{IeronymouOrd2}. In all these articles, the obstruction to weak approximation comes from the fact that the transcendental quaternion algebra has non-constant evaluation at the place at infinity.  With a construction similar to the one used in \cite{HasVarWAorder2}, Hassett and Várilly-Alvarado \cite{HasVarHassePrincipleOrd2} have also built an example of a $2$-torsion element on a K$3$ surface that obstructs the Hasse principle. 

Furthermore, there are examples of transcendental elements of order $3$ on K$3$ surfaces that obstruct the Hasse principle or weak approximation (for example, see \cite{PreuOrd3}, \cite{NewtonOrd3} and \cite{BergVarOrd3}). In all these cases, the evaluation map at the place at infinity has to be trivial, since $\br(\R)$ does not contain elements of order $3$, and the obstruction to weak approximation comes from the evaluation map at the prime $3$, which in every example is a prime of bad reduction for the K$3$ surface taken into account. Therefore, none of the examples mentioned above can be used to give a negative answer to Question \ref{questionSD}. 

\vspace{5mm}

\noindent \textbf{Outline of the paper.} Section $2$ contains the proof of Theorem \ref{thm}. In Section $3$ we show that the K$3$ surface $X$ has good \emph{ordinary} reduction at the prime $2$. Moreover, we explain the ideas behind the construction of the quaternion algebra $\A$ of Theorem \ref{thm} {and why we could expect a priori that it obstructs weak approximation on $X$}. Finally, in Section $4$ we slightly generalise the result, by exhibiting a family of K$3$ surfaces for which there exists a $2$-torsion element in the Brauer group whose evaluation map on $\Q_2$-points is non-constant. 

The computations in Theorem \ref{thm} and \ref{thm2} were done using \emph{Magma} \cite{magma}.

\vspace{5mm}

\noindent \textbf{Acknowledgements.} I am deeply grateful to Martin Bright for introducing me to the topic and for the ideas he shared with me that were very helpful in writing this paper. I thank Francesco Viganò for the useful conversation that helped me in finding the equation defining $X$. I am grateful to the anonymous referees for helpful comments.

\section{Proof of the main theorem}
In the first part of the proof we will show that the element $\A\in \br(k(X))$ lies in $\br(X)$. Next, we will exhibit two points $P_1,P_2\in X(\Q_2)$ such that 
$$\A(P_1)\ne \A(P_2).$$ 
Finally, we will prove that, for every place $\nu$ different from $2$, the evaluation map 
$$|\A|\colon X(\Q_\nu)\rightarrow \br(\Q_\nu)$$
is constant. 
\begin{proof}[Proof of Theorem \ref{thm}]
Let $f:=z^3+w^2x+xyz$ and $C_x$,$C_z$,$C_{f}$ be the closed subsets of $X$ defined by the equations $x=0$, $z=0$ and $f=0$ respectively. The quaternion algebra $\A$ defines an element in $\br(U)$, where $U:=X\setminus (C_x\cup C_z \cup C_{f})$.
The purity theorem for the Brauer group \cite[ Theorem 3.7.2]{SkoTh}, assures us of the existence of the exact sequence 
\begin{equation}\label{purity}
    0 \rightarrow \br(X)[2] \rightarrow \br(U)[2] \xrightarrow{\oplus \partial_{D}} \bigoplus_{D} \h^1(k(D),\Z/2)
\end{equation}
where $D$ ranges over the irreducible divisors of $X$ with support in $X\setminus U$ and $k(D)$ denotes the residue field at the generic point of $D$. 

In order to use the exact sequence (\ref{purity}) we need to understand what the prime divisors of $X$ with support in $X\setminus U=C_x \cup C_z \cup C_{f}$ look like. It is possible to check the following:
\begin{itemize}
    \item $C_x$ has as irreducible components $D_1$ and $D_2$, defined by the equations $\{x=0,z=0\}$ and $\{x=0,y^3+z^2w=0\}$ respectively;
    \item $C_z$ has as irreducible components $D_1$ and $D_3$, where $D_3$ is defined by the equations $\{z=0,x^2y+w^3=0\}$;
    \item $C_{f}$ has as irreducible components $D_1,D_4$ and $D_5$, where $D_4$ and $D_5$ are defined by the equations $\{z^3+xw^2=0,y=0\}$ and $\{y^3z-x^2z^2+y^2w^2=0, x^3+y^2z=0,xyz+z^3+xw^2=0\}$ respectively.
\end{itemize}
Therefore, we can rewrite (\ref{purity}) in the following way:
\begin{equation}\label{purity2}
    0 \rightarrow \br(X)[2] \rightarrow \br(U)[2] \xrightarrow{\oplus \partial_{D_i}} \bigoplus_{i=1}^5 H^1(k(D_i),\Z/2).
\end{equation}
Moreover, we have an explicit description of the residue map on quaternion algebras: for an element $(a,b)\in \br(U)[2]$ we have 
\begin{equation}\label{eqResidueMap}
    \partial_{D_i}(a,b)=\left[(-1)^{\nu_i(a)\nu_i(b)}\frac{a^{\nu_i(b)}}{b^{\nu_i(a)}}\right]\in \frac{k(D_i)^\times}{k(D_i)^{\times2}}\simeq H^1(k(D_i),\Z/2) 
\end{equation}
where $\nu_i$ is the valuation associated to the prime divisor $D_i$. This follows from the definition of the tame symbols in Milnor $K$-theory together with the compatibility of the residue map $\partial_D$ with the tame symbols given by the Galois symbols (see \cite{GilleSzamuely}, Proposition $7.5.1$).

We can proceed with the computation of the residue maps $\partial_{D_i}$ for $i=1,\dots,5$:

\begin{enumerate}
    \item $\nu_1(f)=\nu_1(x)=\nu_1(z)=1.$ Hence,
    $$\partial_{D_1}\left(\frac{f}{x^3},-\frac{z}{x}\right)=\left[\left(-\frac{z}{x}\right)^2\right]=1\in \frac{k(D_1)^\times}{k(D_1)^{\times 2}}.$$
    \item $\nu_2(x)=1$ and $\nu_2(f)=\nu_2(z)=0.$ Hence,
    $$\partial_{D_2}\left(\frac{f}{x^3},-\frac{z}{x}\right)=\left[-\left(\frac{f}{x^3}\right)^{-1}\left(-\frac{z}{x}\right)^3\right]=\left[\frac{z^3}{f}\right]=1\in \frac{k(D_2)^\times}{k(D_2)^{\times 2}}$$
    where the last equality follows from the fact that $x=0$ on $D_2$, thus $f\mid_{D_2}=z^3.$
    \item $\nu_3(x)=\nu_3(f)=0$ and $\nu_3(z)=1$. Hence, 
    $$\partial_{D_3}\left(\frac{f}{x^3},-\frac{z}{x}\right)=\left[\left(\frac{f}{x^3}\right)\right]=\left[ \left(\frac{w}{x}\right)^2\right]=1\in \frac{k(D_3)^\times}{k(D_3)^{\times 2}}$$
    where the last equality follows from the fact that $z=0$ on $D_3$, thus $f\mid_{D_3}=w^2x.$
    \item $\nu_4(x)=\nu_4(z)=0$ and $\nu_4(f)=1$. Hence, 
    $$\partial_{D_4}\left(\frac{f}{x^3},-\frac{z}{x}\right)=\left[\left(-\frac{x}{z}\right)\right]=\left[\left( \frac{z}{w}\right)^2\right]=1\in \frac{k(D_4)^\times}{k(D_4)^{\times 2}}$$
    where the last equality follows from the fact that $z^3+w^2x=0$ on $D_4$, thus $-\frac{x}{z}=\left(\frac{z}{w}\right)^2$.
    \item $\nu_5(x)=\nu_5(z)=0$ and $\nu_5(f)=1$. Hence, 
    $$\partial_{D_5}\left(\frac{f}{x^3},-\frac{z}{x}\right)=\left[\left(-\frac{x}{z}\right)\right]=\left[\left( \frac{y}{x}\right)^2\right]=1\in \frac{k(D_5)^\times}{k(D_5)^{\times 2}}$$
    where the last equality follows from the fact that $x^3+y^2z=0$ on $D_5$, thus $-\frac{x}{z}=\left(\frac{y}{x}\right)^2.$
\end{enumerate}
Therefore, $\partial_{D_i}(\A)=0$ for all $i\in \{1,\dots,5\}$, hence $\A\in \br(X)$.

We now show that the element $\A$ obstructs weak approximation on $X$. Let $\mathcal{X}\subseteq \p^3_\Z$ be the projective scheme defined by the equation 
\begin{equation}\label{eq}
    x^3y+y^3z+z^3w+w^3x+xyzw=0.
\end{equation}
$\mathcal{X}$ is a $\Z$-model for $X$ and has good reduction at the prime $2$.

Let $P_1:=(1:0:1:0)\in \mathcal{X}(\Z_2)$; then $P_1$ is such that $\A(P_1)=(1,-1)$. Therefore $\A(P_1)$ is the trivial class in $\br(\Q_2)$. On the other hand, Hensel's lemma assures us of the existence of a solution $P_2=(1:2:1:d)\in \mathcal{X}(\Z_2)$ whose reduction modulo $8$ is $(1:2:1:2)$. Hence, 
$$\A(P_2)=(f(P_2),-1)\quad \text{ with }\quad f(P_2)\equiv 7 \pmod{8}.$$
Therefore, we get that $\A(P_2)$ defines a non-trivial element in the Brauer group of $\Q_2$
\cite[Theorem 3.1]{ACourseInArithmetic}. The existence of the points $P_1$ and $P_2$ implies that there is a Brauer--Manin obstruction to weak approximation arising from $\A$. Indeed, 
$$X(\Ad_\Q)^{\A}=\Bigg\{(x_p)_p\in X(\Ad_\Q): \sum_p \inv_p \A(x_p)=0 \Bigg\}\subsetneq X(\Ad_\Q).$$

In order to conclude the proof of the theorem we investigate the behaviour of the evaluation map at the other primes and at infinity. For every prime $p$ let $\mathcal{X}_p$ be the base change of $\mathcal{X}$ to $\Z_p$. We distinguish the following cases.

\medskip

\fbox{Case ${p \notin \{3,5,17,\infty\}}$.} In this case, $\mathcal{X}$ has good reduction at $p$. Therefore, we can use \cite[Proposition $2.4$]{colliotskogoodred} to conclude that the evaluation map 
$$|\A|\colon \mathcal{X}(\Z_p)\rightarrow \br(\Q_p)$$
is constant. Moreover, $P=(1:0:1:0)\in \mathcal{X}(\Z_p)$ and 
$$\A(P)=(1,-1)$$
which is trivial in $\br(\Q_p)$; hence the evaluation map is trivial on the whole $\mathcal{X}(\Z_p)=X(\Q_p).$

\medskip

\fbox{Case $p\in \{3,5,17\}$.} Under this assumption, $\mathcal{X}_p/\Z_p$ is not smooth. In these three cases, we want to show that the evaluation map is trivial on $\mathcal{X}(\Z_p)$ by showing that it factors through $\br(\Z_p)$.

The special fibre $Y_p:=\mathcal{X}_p \times_{\Z_p} \spec(\F_p)$ is a non-smooth $\F_p$-scheme. However, $Y_p$ is an irreducible $\F_p$-scheme, with just isolated singularities. The $\Z_p$-points of $\mathcal{X}_p$ are all smooth. In fact, $Y_p$ contains just one singular point defined over $\F_p$ and that does not even lift to a $\Z/p^2 \Z$-point. 
Let $\mathcal{V}$ be the smooth locus of $\mathcal{X}_p$; because of what we have just said we have
$$X(\Q_p)=\mathcal{X}(\Z_p)=\mathcal{V}(\Z_p).$$
Let $V$ be the base change of $\mathcal{V}$ to $\spec(\Q_p)$. The purity theorem on $\mathcal{V}$ \cite[Theorem $3.7.1$]{colliotskogoodred} gives us the exact sequence 
$$\br(\mathcal{V})[2]\rightarrow \br(V)[2]\xrightarrow{\partial_{D_p}}  H^1(k(D_p),\Z/2\Z)$$
where $D_p$ is the divisor associated to the special fibre ($D_p$ is the smooth locus of $Y_p$). We just need to show that $\partial_{D_p}(\A)=0.$
Let $\nu_p$ be the valuation corresponding to the prime divisor $D_p$; then 
$$\nu_p\left(\frac{f}{x^3}\right)=0 \quad \text{ and }\quad \nu_p \left(-\frac{z}{x}\right)=0.$$
Indeed, the point $(1:0:1:0)\in Y_p(\F_p)$ is smooth, hence it lies in $D_p$. Moreover 
$$\frac{f}{x^3}(1:0:1:0)=1 \quad \text{ and } \quad -\frac{z}{x}(1:0:1:0)=-1.$$ Therefore both $\frac{f}{x^3}$ and $-\frac{z}{x}$ do not vanish identically on $D_p$, which implies that
$\partial_{D_p}(\A)=0$. Therefore, $\A$ lies in $\br(\mathcal{V})\subseteq \br(X_p)$ and the evaluation map factors as
\begin{center}
    \begin{tikzcd}
    \mathcal{V}(\Z_p)\arrow[rr,"|\A|"] \arrow[rd] & &\br(\Q_p).\\
    &\br(\Z_p) \arrow[ru]
    \end{tikzcd}
\end{center}
Since $\br(\Z_p)$ is trivial, the evaluation map has to be constant and trivial.

\medskip

\fbox{Case $p=\infty$.} The evaluation map 
$$|\A|\colon X(\R)\rightarrow \br(\R)$$
is constant and equal to $0$.

We will show that it is constant on the dense open subset
$$W:=\{P\in X(\R): x(P),z(P),f(P)\ne 0\}\subseteq X(\R).$$ Then, from the continuity of the evaluation map it will follow that it has to be constant also on the whole of $X(\R)$. Let $P=(\alpha:\beta:\gamma:\delta)\in W$, thus $\gamma \neq 0$. First, assume that $-\frac{\gamma}{\alpha}>0$. Then 
$$\A(P)=\left(\frac{f(P)}{x(P)^3},-\frac{z(P)}{x(P)}\right)=\left(\frac{f(P)}{\alpha^3},-\frac{\gamma}{\alpha}\right)$$
is trivial in $\br(\R)$. Now, suppose that $-\frac{\gamma}{\alpha}<0$. Without loss of generality, we can assume that both $\alpha$ and $\gamma$ are positive. We want to show that in this case $f(P)$ has to be positive: 
\begin{itemize}
    \item if $\delta=0$, then $P\in X(\R)$ implies that  $\beta(\alpha^3+\beta^2\gamma)=0.$ Therefore $\beta=0$, since $\alpha^3+\beta^2\gamma\geq \alpha^3>0.$ Hence, $f(P)=\gamma^3>0$;
    \item if $\delta \ne 0$ then $P\in X(\R)$ implies 
    $$f(P)=-\frac{\beta}{\delta} (\alpha^3+\beta^2\gamma).$$
    Hence, since $\alpha^3+\beta^2\gamma>0$, 
    $$f(P)>0 \quad \text{ if and only if }\quad -\frac{\beta}{\delta}>0.$$ Equivalently, $\beta,\delta$ do not have the same sign. Hence, we just need to show that there is no point $P\in W$ with $\alpha,\gamma$ positive and $\beta,\delta$ with the same sign. First, we observe that $\beta,\delta$ cannot both be positive, since otherwise
    $$\alpha^3\beta+\beta^3\gamma+\gamma^3\delta+\delta^3\alpha+\alpha \beta \gamma \delta >0.$$
    On the other hand $\beta,\delta$ cannot both be negative. Indeed, we have that $P\in X(\R)$ if and only if
    $$\alpha^3(-\beta)+(-\beta)^3\gamma+\gamma^3(-\delta)+(-\delta)^3\alpha=\alpha(-\beta)\gamma(-\delta).$$
    Without loss of generality we may assume that $\alpha\geq \max\{-\beta,\gamma,-\delta\}$; but if $\alpha,-\beta,\gamma,-\delta$ are all positive, then 
    $$\alpha^3(-\beta)+(-\beta)^3\gamma+\gamma^3(-\delta)+(-\delta)^3\alpha>\alpha(-\beta)\gamma(-\delta).$$
    Hence, $(\alpha:\beta:\gamma:\delta)\notin X(\R)$.\qedhere
\end{itemize}

\end{proof}

\section{Construction of the quaternion algebra $\A$}
The aim of this section is to give a glance at the ideas behind the construction of the quaternion algebra $\A$. 

Through this section we will indicate by $\mathcal{X}_2$ the base change of the $\Z$-model $\mathcal{X}$ to $\Z_2$, with $X_2$ the base change of $\mathcal{X}_2$ to $\Q_2$ and with ${Y}$ the reduction of $\mathcal{X}_2$ at the prime $2$.

In the previous section we already mentioned that the K$3$ surface $X$ has good reduction at the prime $2$ (i.e. $\mathcal{X}_2$ is a smooth $\Z_2$-scheme); actually $X$ has good \emph{ordinary} reduction at the prime $2$, as we will show in the following lemma.
\begin{lemma}\label{lemmaYordinary}
$Y$ is an ordinary K$3$ surface over $\F_2$.
\end{lemma}
\begin{proof}

It is enough to show that the cardinality of $Y(\F_2)$ is even (see \cite{taelman2018ordinary}). $Y$ is the projective $\F_2$-variety defined by the equation 
$$x^3y+y^3z+z^3w+w^3x+xyzw=0.$$
Therefore, it is possible to compute $|Y(\F_2)|$ directly, which turns out to be equal to $10$. 
\end{proof}
Before proceeding with the actual construction, we mention a property of the K$3$ surface $Y$, related to the fact that it is ordinary. For every $q\geq 0$, let $\Omega^q_{Y/\F_2}$ be the sheaf of differential $q$-forms on $Y$. We define 
$$Z_{Y/\F_2}^q:=\ker\Big(d:\Omega^q_{Y/\F_2}\rightarrow \Omega^{q+1}_{Y/\F_2}\Big) \quad \text{and} \quad B_{Y/\F_2}^q:=\im\Big(d:\Omega^{q-1}_{Y/\F_2}\rightarrow \Omega^q_{Y/\F_2}\Big).$$
Let 
$$\text{C}_{Y/\F_p}\colon Z_{Y/\F_2}^{\bullet} \rightarrow \Omega^{\bullet}_{Y/\F_2}$$ be the Cartier operator on $Y$ (for the construction of the Cartier operator see \cite[\textsection $2$]{Illusie}). For every $q$ the sheaf of logarithmic differential $q$-forms $\Omega_{Y/\F_2,\log}^q$ is defined as the kernel of the morphism 
$$1-\text{C}_{Y/\F_2}\colon Z_{Y/\F_2}^q\rightarrow \Omega^q_{Y/\F_2}.$$
Hence, $\Omega^q_{Y/\F_2,\log}$ fits in the following exact sequence:
\begin{equation}\label{eqSeqLog}
    0\rightarrow \Omega^q_{Y/\F_2,\log} \rightarrow Z_{Y/\F_2}^q\xrightarrow{1-C_{Y/\F_2}} \Omega^q_{Y/\F_2}.
\end{equation}

\begin{rmk}
The Cartier operator is defined, more generally, for smooth $S$-schemes $X$ in characteristic $p$. The Cartier operator $\text{C}_{X/S}$ goes from $Z_{X/S}^{\bullet}$ to $\Omega_{X^{(p)}/S}^{\bullet}$, where $X^{(p)}$ is the base change of $X$ by the Frobenius morphism $F_S:S\rightarrow S$. In this setting $\Omega^q_{X/S,\log}$ is defined as the kernel of the morphism 
$$W^*-\text{C}_{X/S}\colon Z^q_{X/S}\rightarrow \Omega^q_{X^{(p)}/S}$$
where $W^*$ is the map induced on the differential forms by the natural projection $$W\colon X^{(p)}\rightarrow X.$$ However, since in our case $S=\F_2$, we have $Y^{(2)}=Y$ and $W=\text{id}$.
\end{rmk}
The sheaf $\Omega^q_{Y/\F_2,\log}$ is a sheaf of $\F_2$-vector spaces on the Zariski site of $Y$. Moreover, if we look at it on $Y_\et$, then its formation is compatible with étale base change \cite[2.1.8]{Illusie}.

Let $k$ be an algebraic closure of $\F_2$ and $\overline{Y}$ the base change of $Y$ to $k$. Bloch and Kato proved \cite[Proposition 7.3]{BlochKatoEtale} that $Y$ ordinary implies that the natural map
$$\h^0\Big(\overline{Y},\Omega^2_{\overline{Y}/k,\log}\Big)\otimes_{\F_2} k\rightarrow \h^0\Big(\overline{Y},\Omega^2_{\overline{Y}/k}\Big)$$
is an isomorphism. For K$3$ surfaces, $\h^0\left(\overline{Y},\Omega^2_{\overline{Y}/k}\right)$ is a one-dimensional $k$-vector space. Hence $\h^0\left(\overline{Y},\Omega^2_{\overline{Y}/k,\log}\right)$ has to be a one-dimensional $\F_2$-vector space. Let $\omega$ be the only non-trivial element in $\h^0\Big(\overline{Y},\Omega^2_{\overline{Y}/k,\log}\Big)$. Since $C_{\overline{Y}/k}$ respects the Galois action, we get that the element $\omega$ is Galois fixed, hence it comes from $\h^0\Big(Y,\Omega^2_{Y/\F_2}\Big)$. Therefore, the unique non-trivial element of $\h^0\Big(Y,\Omega^2_{Y/\F_2}\Big)$ must be logarithmic. 
\begin{lemma}\label{lemmaLogF}
Let $F$ be the function field of $Y$. The image of $\omega\in \h^0\Big(Y,\Omega^2_{Y/\F_2}\Big)$ in $\Omega^2_{F/\F_2}$ can be written as 
$$\frac{d\eta_1}{\eta_1}\wedge \frac{d\eta_2}{\eta_2}, \quad \text{where}\quad \eta_1=\frac{z^3+w^2x+xyz}{x^3} \quad \text{and} \quad \eta_2=\frac{z}{x}.$$
\end{lemma}

\begin{proof}

Let $\xi\in Y$ be the generic point of $Y$ and $\omega_\xi$ the image of $\omega$ in $\Omega^2_{F/\F_2}$ under the inclusion 
$$\h^0\Big(Y,\Omega^2_{Y/\F_2}\Big)\hookrightarrow \Omega^2_{F/\F_2}.$$
{For convenience, to give an explicit description of a non-zero element in $\h^0\Big(Y,\Omega^2_{Y/\F_2}\Big),$ we will use the following notation: instead of the variables $\{x,y,z,w\}$, for the equation defining $Y$ we will use the variables $\{x_0,x_1,x_2,x_3\}$ and we denote by $G(x_0,x_1,x_2,x_3)$ the polynomial defining $Y$.}

 For every permutation $\{p,q,i,j\}$ of $\{0,1,2,3\}$ we define $W_{p,q}\subseteq Y$
 as the open subset of $Y$ where $x_p\cdot  \frac{\partial G}{\partial x_q}$ does not vanish. Moreover, we set
$$\omega_{p,q}:=\frac{d\left(\frac{x_i}{x_p}\right)\wedge d\left(\frac{x_j}{x_p}\right)}{\frac{1}{x_p^3}\cdot \frac{\partial G}{\partial x_q}}\in \h^0(W_{p,q},\Omega^2_{Y/\F_2}).$$
Since $Y$ is smooth, the open sets $\{W_{p,q}\}$ cover it. It is easy to check that, since we are working over the field $\F_2$, for every $(p,q)\ne (p',q')$ 
$$\omega_{p,q}\mid_{W_{p,q}\cap W_{p',q'}} = \omega_{p',q'}\mid_{W_{p,q}\cap W_{p',q'}}.$$
Therefore, there exists $\omega\in \h^0\Big(Y,\Omega^2_{Y/\F_2}\Big)$ such that for every $p,q$ as above
$$\omega\mid_{W_{p,q}}=\omega_{p,q}\in \h^0\Big(W_{p,q},\Omega^2_{Y/\F_2}\Big).$$
{Now, going back to the usual notation with which we denote our variables by $\{x,y,z,w\},$ let $G_w$ be the partial derivative of the polynomial defining $Y$ with respect to the variable $w$.}
We have
$$\omega_\xi=(\omega_{0,3})_\xi=\frac{d\left(\frac{y}{x}\right)\wedge d\left(\frac{z}{x}\right)}{\frac{G_w}{x^3}}=\frac{d\left(\frac{G_w}{x^3}\right)\wedge d\left(\frac{z}{x}\right)}{\frac{G_w}{x^3}\cdot \frac{z}{x}}.$$
Indeed, 
\begin{align*}
    &d\Bigg(\frac{G_w}{x^3}\Bigg)\wedge d\Bigg(\frac{z}{x}\Bigg)=d\Bigg(\frac{z^3+w^2x+xyz}{x^3}\Bigg) \wedge d\Bigg(\frac{z}{x}\Bigg)=d\Bigg(\frac{z^3}{x^3}+\frac{w^2}{x^2}+\frac{yz}{x^2}\Bigg)\wedge d\Bigg(\frac{z}{x}\Bigg)
\end{align*}
Using that $d\left(\frac{z}{x}\right)\wedge d\left(\frac{z}{x}\right)=0$, together with the fact that we are working over a field of characteristic $2$, we get 
$$d\Bigg(\frac{G_w}{x^3}\Bigg)\wedge d\Bigg(\frac{z}{x}\Bigg)=\frac{z}{x}\cdot d\Bigg(\frac{y}{x}\Bigg) \wedge d\Bigg( \frac{z}{x}\Bigg).$$
\end{proof} 
We can finally proceed with the construction of the quaternion algebra $\A$ and explain why we could expect a priori that $\A$ gives an obstruction to weak approximation on $X$. 

Let $R$ be the henselisation of the discrete valuation ring $\Os_{\mathcal{X}_2,Y}$ and $K^h$ be the fraction field of $R$. Bloch and Kato \cite{BlochKatoEtale} introduced a decreasing filtration 
$$\Big\{U^m\h^2\big(K^h,\mu_2^{\otimes 2}\big)\Big\}_{m\geq 0}$$ 
on $\h^2\big(K^h,\mu_2^{\otimes 2}\big)$ as follows. For $a_1,a_2\in K^h$, let $(a_1,a_2)_2$ denote the class $$\delta(a_1)\cup \delta(a_2)\in \h^2\big(K^h,\mu_2^{\otimes 2}\big)$$
where $\delta\colon \big(K^h\big)^\times \rightarrow \h^1\big(K^h,\mu_2\big)$ is the connecting map coming from the Kummer sequence. Let $U^0 \text{H}^2\big(K^h,\mu_2^{\otimes 2}\big):=\h^2\big(K^h,\mu_2^{\otimes 2}\big)$; for $m>0$, let $U^m \h^2\big(K^h,\mu_2^{\otimes 2}\big)$ be the subgroup of $\h^2\big(K^h,\mu_2^{\otimes 2}\big)$ generated by the elements of the form $\big(1+a_1 2^m,a_2\big)_2$, with $a_1\in R$ and $a_2\in \big(K^h\big)^\times$. Moreover, for every $m\geq 0$ let 
$$\text{gr}^m:= \frac{U^m\h^2\big(K^h,\mu_2^{\otimes2}\big)}{U^{m+1}\h^2\big(K^h,\mu_2^{\otimes 2}\big)}.$$

In \cite[\textsection $5$]{BlochKatoEtale}, Bloch and Kato proved that the map
\begin{align*}
    \rho_0\colon \Omega^2_{F,\log} \oplus \Omega^1_{F,\log} &\rightarrow \text{gr}^0:= \frac{\h^2\big(K^h,\mu_2^{\otimes2}\big)}{U^1\big(\h^2\big(K^h,\mu_2^{\otimes 2}\big)\big)}\\
    \left(\frac{d\eta_1}{\eta_1}\wedge \frac{d\eta_2}{\eta_2},0\right)&\mapsto ({\tilde{\eta}}_1,{\tilde{\eta}}_2)_2\\
    \left(0,\frac{d\eta_1}{\eta_1}\right) &\mapsto (\tilde{\eta}_1,2)_2
\end{align*}
is an isomorphism, where $\tilde{\eta}_1,\tilde{\eta}_2$ are arbitrary lifts of $\eta_1,\eta_2$ to $K^h$. 
\begin{rmk}
We are working over $\Q_2$, which contains a primitive second root of unity, hence we have an isomorphism \cite[Proposition $4.7.1$]{GilleSzamuely}
$$\h^2\big(K^h, \mu_2^{\otimes 2}\big)\simeq  \br\big(K^h\big)[2]$$
which sends $(a,b)_2$ to the class of the quaternion algebra $(a,b)$.
\end{rmk}
Let $\mathcal{B}\in \br(X_2)[2]$; we can always look at $\mathcal{B}$ as an element in $\text{H}^2\big(K^h,\mu_2^{\otimes 2}\big)$. Hence, there exists an $m\geq 0$ such that $\mathcal{B}$ gives a non-zero element in $ \text{gr}^m$. Bright and Newton proved \cite{https://doi.org/10.48550/arxiv.2009.03282} that knowing such an $m$ gives information about the behaviour of the evaluation map 
\begin{align*}
    |\mathcal{B}|\colon\mathcal{X}_2(\Os_L)&\rightarrow \br(L)\\
    P &\mapsto \mathcal{B}(P)
\end{align*} for finite field extensions $L/\Q_2$. We will make this sentence more precise in the following subsection. 
\subsection{The evaluation filtration}\label{subsectionEvFiltration}

Bright and Newton define an evaluation filtration on the Brauer group of $X_2$ in the following way. Given a finite field extension $L$ of $\Q_2$, let $e_{L/\Q_2}$ be the ramification index of $L$, $\Os_L$ its ring of integers and $\pi$ a uniformiser. For all positive integers $r$ and $P\in \mathcal{X}_2$, let $B(P,r)$ be the set of points $Q\in \mathcal{X}_2(\Os_L)$ such that $Q$ has the same image as $P$ in $\mathcal{X}_2(\Os_L/\pi^r)$ (equivalently, we will say that $Q \equiv P \pmod{\pi^r}$). We define
 \begin{alignat*}{2}
     \Ev_n \br X_2 :=\{ \mathcal{B} \in \br(X_2) &\mid \forall L/\Q_2 \text{ finite, } \forall P\in \mathcal{X}_2(\Os_L)\\
     &|\mathcal{B}| \text{ is constant on }B(P,e_{L/\Q_2} n+1)\}, \qquad (n\geq 0)\\ 
     \Ev_{-1} \br X_2 :=\{ \mathcal{B} \in \br(X_2&) \mid \forall L/\Q_2 \text{ finite, } |\mathcal{B}| \text{ is constant on }\mathcal{X}_2(\Os_L)\}\\
     \Ev_{-2} \br X_2 :=\{ \mathcal{B} \in  \br(X_2&) \mid \forall  L/\Q_2  \text{ finite, } |\mathcal{B}| \text{ is zero on }\mathcal{X}_2(\Os_L)\}
 \end{alignat*}

In order to compare the evaluation filtration with the filtration $U^m\h^2\big(K^h,\mu_2^{\otimes 2}\big)$, Bright and Newton used the filtration $\{\fil_n \br(X_2)[2]\}$ on $\br(X_2)[2]$ given by Kato's Swan conductor \cite[\textsection $2$]{https://doi.org/10.48550/arxiv.2009.03282}. Briefly, using the Swan conductor, Kato \cite{Kato} defined an increasing filtration on $\h^2\big(K^h,\Z/2\Z(1)\big)$. From the Kummer sequence, $\h^2\big(K^h,\Z/2\Z(1)\big)$ is isomorphic to $\br\big(K^h\big)[2]$. Therefore, we get a filtration $\{\fil_n\br\big(K^h\big)[2]\}_{n\geq 0}$ on $\br\big(K^h\big)[2]$; the pullback of this filtration gives us a filtration $\{\fil_n \br(X_2)[2]\}_{n\geq 0}$ on $\br(X_2)[2]$.

Moreover, since $\Q_2$ contains a primitive second root of unity, we have that $\Z/2\Z(1)\simeq \mu_2^{\otimes 2}$, thus we can identify 
\begin{equation}\label{eqIso}
    \text{H}^2\big(K^h,\Z/2\Z(1)\big)\simeq \text{H}^2\big(K^h,\mu_2^{\otimes 2}\big).
\end{equation}
 Kato \cite[Lemma $4.3$]{Kato}, proved that in our setting the isomorphism of equation (\ref{eqIso}) induces isomorphisms
\begin{equation}\label{eqIsoFiltration}
    U^m \text{H}^2\big(K^h,\mu_2^{\otimes 2}\big) \simeq \fil_{2-m} \br\big(K^h\big)[2], \qquad 0\leq m\leq 2.
\end{equation}
Moreover, for $m>2$ we have that $U^m\text{H}^2\big(K^h,\mu_2^{\otimes 2}\big)$ vanishes.

In particular,  using the filtration $U^m$ on $\text{H}^2\big(K^h,\mu_2^{\otimes 2}\big)$ we can describe $\fil_0\br(X_2)[2]$ as
$$\Big\{\mathcal{B}\in \br(X_2)[2] \text{ such that the image of }\mathcal{B} \text{ in }\text{H}^2\big(K^h,\mu_2^{\otimes 2}\big) \text{ lies in } U^2 \text{H}^2\big(K^h,\mu_2^{\otimes 2}\big)\Big\}.$$ 
Finally, Bright and Newton proved \cite[Theorem A]{https://doi.org/10.48550/arxiv.2009.03282} that there is an equality
$$\Ev_0 \br(X_2) =\fil_0 \br(X_2).$$
\subsection{The quaternion algebra $\A$}
Summing up, the construction of $\A$ goes through the following steps. By Lemma \ref{lemmaLogF} we know that $\omega\in \h^0\big(Y,\Omega^2_{Y/\F_2}\big)$ is such that
$$\omega_\xi=\frac{d\eta_1}{\eta_1}\wedge \frac{d\eta_2}{\eta_2}\in \Omega^2_{F,\log}$$
is a non-trivial logarithmic $2$-form of $F$, where $\eta_1:=\frac{G_w}{x^3},\eta_2:=\frac{z}{x}$ The isomorphism $\rho_0$ assures us that for every choice of lifts $\tilde{\eta}_1,\tilde{\eta}_2$ in $K^h$, the element $(\tilde{\eta}_1,\tilde{\eta}_2)\in \br\big(K^h\big)[2]$ has non-trivial image in $\text{gr}^0$, that is
$(\tilde{\eta}_1,\tilde{\eta}_2)_2$ lies in ${\text{H}^2\big(K^h,\mu_2^{\otimes 2}\big)}\setminus{U^1 \text{H}^2\big(K^h,\mu_2^{\otimes 2}\big)}.$
At this point, the idea behind the construction of $\A$ is to find lifts $\tilde{\eta}_1,\tilde{\eta}_2$ such that $(\tilde{\eta}_1,\tilde{\eta}_2)$ defines an element in $\br(X)$. Indeed, if such lifts exist, then the image of $(\tilde{\eta}_1,\tilde{\eta}_2)$ in $\br(X_2)$ does not lie in $\fil_0\br(X_2)$. Rather surprisingly, as proven in Theorem \ref{thm}, it turns out that the choice of lifts $$\tilde{\eta}_1:=\frac{z^3+w^2x+xyz}{x^3} \quad \text{and} \quad \tilde{\eta}_2:=-\frac{z}{x}$$ 
defines an element 
$$\mathcal{A}:=(\tilde{\eta}_1,\tilde{\eta}_2)\in \br(X).$$
By construction, using equation (\ref{eqIsoFiltration}), the image of $\A$ in $\br\big(K^h\big)[2]$ is not in $\fil_1 \br\big(K^h\big)[2]$. Hence, if we look at $\A$ in $\br(X_2)$, then it does not lie in $\fil_1 \br(X_2)\supseteq \fil_0\br(X_2)$. In particular, by \cite[Theorem A$(3)$]{https://doi.org/10.48550/arxiv.2009.03282} we have that there exists a finite field extension $L/\Q_2$ and two points $P,Q\in \mathcal{X}_2(\Os_L)$ such that $P$ and $Q$ have the same image in $\mathcal{X}_2(\Os_L/\pi_L)$ and $\A(P)\ne \A(Q).$ We saw in the proof of Theorem \ref{thm} that there exist two points $P_1$ and $P_2$ defined over $\Z_2$, with the same reduction modulo $2$ and whose evaluation map is different. Namely, there is no need to take a field extension of $\Q_2$ in our case. 

\begin{rmk}
From the identification 
$$\br\big(K^h\big)[2]\simeq U^0\big(\h^2\big(K^h,\mu_2^{\otimes 2}\big)\big)\simeq \fil_2 \br\big(K^h\big)[2]$$
we get that 
$$\br(X_2)[2]=\fil_2 \br(X_2)[2].$$
Hence, clearly, $\A\in \fil_2 \br(X_2)[2].$

\cite[Theorem A$(4)$]{https://doi.org/10.48550/arxiv.2009.03282} tells us that 
\begin{equation}\label{eqEv1}
    \Ev_1 \br(X_2)[2]=\{\mathcal{B}\in \fil_2 \br(X_2)[2]\mid \rsw_{2,2}(\A)\in \Omega_F^2\oplus 0\}.
\end{equation}
Hence, if $\A$ is such that $\rsw_{2,2}(\A)=(\alpha,0)$, then $\A$ lies in $\Ev_1 \br(X_2)[2]$. The notation $\rsw_{2,2}$ denotes the \emph{refined Swan conductor} (for the definition of it see \cite[\textsection $5$]{Kato}). In particular, in our case
$$\rsw_{2,2}:\frac{\br\big(K^h\big)[2]}{\fil_1\br\big(K^h\big)[2]}\rightarrow \Omega^2_F\oplus \Omega^1_F.$$
Furthermore, the refined Swan conductor morphism is strictly related to the morphism $\rho_0$. Indeed, by Kato \cite[Lemma 4.3]{Kato}, we know that
$$\frac{\br\big(K^h\big)[2]}{\fil_1\br\big(K^h\big)[2]}\simeq \frac{\br\big(K^h\big)[2]}{U^1\br\big(K^h\big)[2]}=\text{gr}^0$$
and 
\begin{equation}\label{eqrswrho}
    \rsw_{2,2}(\rho_0(\alpha,\beta))=(\alpha,\beta).
\end{equation}
By construction, in our case, the image of $\A$ in $\text{gr}^0$
is of the form 
$$\rho_0\left(\frac{d\eta_1}{\eta_1}\wedge\frac{d\eta_2}{\eta_2},0\right).$$
Therefore $\rsw_{2,2}(\A)$ belongs to $\Omega^2_F\oplus 0$ and
because of equation (\ref{eqEv1}) it holds that $\A$ lies in $\Ev_1 \br(X_2)[2]$. Thus, by the definition of the last object, the evaluation map from $\mathcal{X}_2(\Z_2)$ to $\br(\Q_2)$ depends only on the reduction of the points modulo $4$. 
\end{rmk}
\begin{rmk}\label{rmkalphap0}
Using Theorem B(4) in \cite{https://doi.org/10.48550/arxiv.2009.03282} we could already predict that the evaluation map attached to $\A$ is non-constant on the $2$-adic points. By equation (\ref{eqrswrho}), we have that $$\rsw_{2,2}=(\omega,0).$$ Let $P_0\in Y(\F_2)$ be the point $(1:0:1:0)$; locally in a neighbourhood of $P_0$, $\omega$ is of the form $$\frac{x^3}{G_w}\cdot d \left(\frac{y}{x}\right)\wedge d\left(\frac{z}{x}\right)=\frac{x^3}{G_w}\cdot d \left(\frac{y}{x}\right)\wedge d\left(\frac{z}{x}-1\right)$$
with $G_w=z^3+w^2x+xyz$ (see the proof of Lemma \ref{lemmaLogF}). The functions $\frac{y}{x}$ and $\frac{z}{x}-1$ constitute a system of parameters for the local ring $\Os_{Y,P_0}$. Hence, we have that
$$\omega_{P_0}=\left(\frac{y}{x}\right)\wedge \left(\frac{z}{x}-1\right)\ne 0 \in \Omega^2_{Y,P_0}.$$
By \cite[Theorem B(4)]{https://doi.org/10.48550/arxiv.2009.03282} there exists a point $Q\in \mathcal{X}(\Z_2)$ whose reduction modulo $2$ coincides with $P_0$ and such that the evaluation map attached to $\A$ maps $B(Q,1)$ surjectively to $\br(\Q_2)[2]$.
\end{rmk}
\section{A family of K3 surfaces with the same property}
In this section we will show that the first part of Theorem \ref{thm} can be easily generalised to a family of K$3$ surfaces that share some properties with our K$3$ surface $X$. 

Let $a,b,c,d,e$ be odd integers, $\underline{\alpha}=(a,b,c,d,e)$ and $X_{\underline{\alpha}}$ be the K$3$ surface in $\p^3_\Q$ associated to the equation
\begin{equation}\label{EqGeneral}
    a\cdot x^3 y+ b\cdot y^3 z+ c\cdot z^3 w+ d \cdot w^3 x+ e\cdot xyzw=0.
\end{equation}
    Let $\mathcal{X}_{\underline{\alpha}}$ be the projective scheme over $\Z$ defined by the polynomial equation (\ref{EqGeneral}). Then $\mathcal{X}_{\underline{\alpha}}$ is a $\Z$-model of $X_{\underline{\alpha}}$. Moreover, since $a,b,c,d,e$ are all odd integers, all these varieties have the same reduction, which we will denote by $Y$, modulo the prime $2$. Hence, by Lemma \ref{lemmaYordinary}, all these K$3$ surfaces have ordinary good reduction at the prime $2$. Therefore, as already pointed out in Section $3$, the unique non-trivial element $\omega$ in $\h^0(Y,\Omega^2_Y)$ must be logarithmic. A natural question that arises at this point is if also for all these K$3$ surfaces there exists an element $\A\in \br(X_{\underline{\alpha}})[2]$ such that $$\rho_0(\omega,0)=[\A]\in\frac{\br(K^h)[2]}{U^1\br(K^h)[2]}.$$
    Indeed, by Section $3.2$ this would imply that, at least after taking a finite field extension of $\Q_2$, the quaternion algebra $\A$ gives an obstruction to weak approximation. The following theorem gives a partial answer to the question. 
\begin{thm}\label{thm2}
Assume that $a b c d \in \Q^{\times 2}$. Then, the class of the quaternion algebra
$$\A=\left(d\cdot \frac{c\cdot z^3+d\cdot w^2 x+e\cdot xyz}{x^3},-(cd)\cdot \frac{z}{x}\right)\in \br(\Q(X_{\underline{\alpha}}))$$
defines an element in $\br(X_{\underline{\alpha}})$. The evaluation map $|\A|\colon X_{\underline{\alpha}}(\Q_2)\rightarrow \br(\Q_2) $ is non-constant, and therefore gives an obstruction to weak approximation on $X$.
\end{thm}
\begin{proof}
The proof is very similar to the first part of the proof of Theorem \ref{thm}. We denote by $f$ the polynomial $c\cdot z^3+d\cdot w^2 x+e\cdot xyz$. Also in this case, let $C_x,C_z,C_f$ be the closed subsets of $X_{\underline{\alpha}}$ defined by the equations $x=0,z=0$ and $f=0$ respectively. Let $U$ be the open subset of $X$ defined as the complement of $C_x\cup C_z\cup C_f$. Clearly, $\A\in \br(U)$. Moreover, 
\begin{itemize}
    \item $C_x$ consists of two irreducible components, $D_1=\{x=0,z=0\}$ and $D_2=\{x=0,b\cdot y^3+c\cdot z^2 w=0\}$.
    \item $C_z$ consists of two irreducible components, $D_1$ and $D_3=\{z=0,a\cdot x^2y+d \cdot w^3=0\}$.
    \item $C_f$ consists of three irreducible components, $D_1$, $D_4=\{y=0, c\cdot z^3+d\cdot w^2x=0\}$ and $D_5=\{f=0,a\cdot x^3+b\cdot y^2z=0,
        be\cdot y^3z - a c\cdot x^2z^2 + bd\cdot y^2w^2=0\}$.
\end{itemize}
In order to show that the quaternion algebra $\A$ lies in the Brauer group of $X$ we will use the exact sequence (\ref{purity2}) coming from the purity theorem and the explicit description of the residue map given in equation (\ref{eqResidueMap}). We will denote by $\nu_i$ the valuation associated to the prime divisor $D_i$.
\begin{enumerate}
    \item $\nu_1(f)=\nu_1(x)=\nu_1(z)=1$, and so $\nu_1\left(\frac{f}{x^3}\right)=-2$ and $\nu_1\left(-\frac{z}{x}\right)=0$. Hence
    $$\partial_{D_1}(\A)=\left[(-1)^{0}\left(d\cdot \frac{f}{x^3}\right)^{0}\left(-\frac{1}{cd} \cdot \frac{x}{z}\right)^{-2}\right]=\left[\left(\frac{1}{cd} \cdot \frac{z}{x}\right)^2\right]=1\in \frac{k(D_1)^\times}{k(D_1)^{\times 2}}.$$
    \item $\nu_2(f)=\nu_2(z)=0$ and $\nu_2(x)=1$, and so $\nu_2\left(\frac{f}{x^3}\right)=-3$ and $\nu_2\left(-\frac{z}{x}\right)=-1$. Hence
    $$\partial_{D_2}(\A)=\left[(-1)^{3}\left(d \cdot \frac{f}{x^3}\right)^{-1}\left(-\frac{1}{cd}\cdot \frac{x}{z}\right)^{-3}\right]=\left[\left(\frac{(cd)^3}{d}\frac{x^3}{c\cdot z^3}\cdot \frac{z^3}{x^3}\right)\right]=1\in \frac{k(D_2)^\times}{k(D_2)^{\times 2}}$$
    where the second equality follows from the fact that $f\mid_{D_2}=c\cdot z^3$.
    \item $\nu_3(f)=\nu_3(x)=0$ and $\nu_3(z)=1$, and so $\nu_3\left(\frac{f}{x^3}\right)=0$ and $\nu_3\left(-\frac{z}{x}\right)=1$. Hence
    $$\partial_{D_3}(\A)=\left[(-1)^{0}\left(d\cdot \frac{f}{x^3}\right)^{1}\left(-\frac{1}{cd} \cdot \frac{x}{z}\right)^{0}\right]=\left[d\cdot \frac{f}{x^3}\right]=1\in \frac{k(D_3)^{\times}}{k(D_3)^{\times 2}}$$
    where the last equality follows from the fact that $f\mid_{D_3}=d\cdot w^2x$.
    \item $\nu_4(f)=1$ and $\nu_4(x)=\nu_4(z)=0$, and so $\nu_4\left(\frac{f}{x^3}\right)=1$ and $\nu_4\left(-\frac{z}{x}\right)=0$. Hence
    $$\partial_{D_4}(\A)=\left[(-1)^{0}\left(d\cdot \frac{f}{x^3}\right)^{0}\left(-\frac{1}{cd} \cdot \frac{x}{z}\right)^{1}\right]=\left[\frac{1}{cd}\cdot \frac{c}{d}\right]=1\in \frac{k(D_4)^\times}{k(D_4)^{\times 2}}$$
    where the second equality follows from the fact that $-\frac{z}{x}=\frac{d}{c}\left(\frac{w}{z}\right)^2$ on $D_4$.
    \item $\nu_5(f)=1$ and $\nu_5(x)=\nu_5(z)=0$, and so $\nu_5\left(\frac{f}{x^3}\right)=1$ and $\nu_5\left(-\frac{z}{x}\right)=0$. Hence
    $$\partial_{D_5}(\A)=\left[(-1)^{0}\left(d \cdot \frac{f}{x^3}\right)^{0}\left(-\frac{1}{cd} \cdot \frac{x}{z}\right)^{1}\right]=\left[\frac{b}{acd}\right]=1\in \frac{k(D_5)^\times}{k(D_5)^{\times 2}}$$
    where the last equality follows from the fact that $-\frac{z}{x}=\frac{a}{b}\left(\frac{x}{y}\right)^2$ on $D_5$ and the assumption that $abcd$ is a square in $\Q$.
\end{enumerate}
The above computations together with the purity theorem show indeed that $\A$ lies in $\br(X_{\underline{\alpha}})$. Finally, we need to show that the evaluation map on the $\Q_2$-points of $X_{\underline{\alpha}}$ is non-constant. Let
$$P_1:=(1:0:1:0)\in X_{\underline{\alpha}}(\Q).$$ Then, $P_1$ is such that $\A(P_1)=(dc,-dc)$,
which is trivial in $\br(\Q_2)$. Furthermore, let $$P_2:=\left(cd:y:1:-2\cdot \frac{acde}{2c+cd} \right)\in X_{\underline{\alpha}}(\Q_2)$$ be such that the reduction modulo $8$ of $y$ is equal to $2 \cdot de$. Then $$f(P_2)\equiv c+d\cdot 4 \cdot (cd) +e \cdot (cd)\cdot 2\cdot de \equiv 7\cdot c \mod 8 $$
and therefore evaluation of $\A$ at $P_2$ is
$$\A(P_2)=\left(d\cdot \frac{f(P_2)}{(cd)^3},-cd\frac{1}{cd}\right)=(g(P_2),-1) \quad \text{with} \quad g(P_2)\equiv 7 \mod{8}.$$
Thus, $\A(P_2)$ defines a non-trivial element in $\br(\Q_2)$. Hence, the element $\A\in \br(X_{\underline{\alpha}})$ gives an obstruction to weak approximation on $X_{\underline{\alpha}}$. 
\end{proof}
    A natural question that arises at this point is what happens if $\Delta:=abcd$ is not a square in $\Q$. Note that, in this case we can repeat the same computations that we did in the proof of Theorem \ref{thm2}. That is, for every divisor $D\ne D_5$ we get $\partial_D(\A)=1$,
    while for $D_5$ we have
    $$\partial_{D_5}(\A)=[\Delta]\in \frac{k(D_5)^\times}{k(D_5)^{\times 2}}.$$
    Hence, in this case, $\A$ defines an element in the Brauer group of the base change of $X_{\underline{\alpha}}$ to $\Q(\sqrt{\Delta})$. With an argument similar to the one of Remark \ref{rmkalphap0}, also in this case we expect to be able to find two points $P_1,P_2$ defined over $\Q_2(\sqrt{\Delta})$ such that $\A(P_1)\ne \A(P_2)$.
\subsection{Final considerations}
As already mentioned in the introduction, this paper was strongly inspired by the following result proven by Bright and Newton.
\begin{thm}\label{thmC}\emph{\cite[Theorem C]{https://doi.org/10.48550/arxiv.2009.03282}}
Let $V$ be a smooth, proper variety over a number field $L$ with $H^0(V,\Omega^2_V)\ne 0$. Let $\mathfrak{p}$ be a prime of $L$ at which $V$ has good ordinary reduction, with residue characteristic $p$. Then there exists a finite field extension $L'/L$, a prime $\mathfrak{p}'$ of $L'$ lying over $\mathfrak{p}$, and an element $\A\in \br V_{L'} \{p\}$ such that the evaluation map $|\A|:V(L'_{\mathfrak{p}'})\rightarrow \br(L'_{\mathfrak{p}'})$ is non-constant. In particular, if $V(\Ad_{L'})\ne \emptyset$ then $\A$ obstructs weak approximation on $V_{L'}$.
\end{thm}
In our example, 
$$V=X=\proj\left(\frac{\Q[x,y,z,w]}{x^3y+y^3z+z^3w+w^3x+xyzw}\right)\subseteq \p^3_\Q$$
is a smooth projective variety defined over the number field $\Q$. 

Since $X$ is a K$3$ surface, the hypotheses of Theorem \ref{thmC} are satisfied. In this example, we were able to construct an element $\A$ that satisfies Theorem \ref{thmC} which is already defined over the rational numbers and the corresponding evaluation map is non-constant on the $2$-adic points. Moreover, $\A$ does not just lie in the $2$-primary part of the Brauer group of $X$, it has order exactly $2$.

It is still unclear whether one can hope to extend this strategy to a more general setting.
\newpage
\bibliographystyle{plain}
\bibliography{bib.bib}
\textsc{Mathematisch Instituut, Niels Bohrweg 1, 2333 CA Leiden, Netherlands}\\
\textit{Email Address:} \textbf{m.pagano@math.leidenuniv.nl}
\end{document}